\documentclass[11pt]{amsart}
\usepackage{amsfonts}
\usepackage{amsmath}
\usepackage{amssymb}
\usepackage{color}
\usepackage{graphicx}
\usepackage{xspace}
\usepackage{axodraw}

\newtheorem{thm}{Theorem}

\newtheorem{lem}[thm]{Lemma}
\newtheorem{prop}[thm]{Proposition}
\newtheorem{defn}{Definition}

\newcommand*{\sh}{{\,\llcorner\!\llcorner\!\!\!\lrcorner\,}}

\newcommand*{\un}{{\mathbf 1}}
\def\shuff#1#2{\mathbin{
      \hbox{\vbox{\hbox{\vrule \hskip#2 \vrule height#1 width 0pt}\hrule}\vbox{\hbox{\vrule \hskip#2 \vrule height#1 width 0pt\vrule }\hrule}}}}
\def\shuffl{{\mathchoice{\shuff{5pt}{3.5pt}}{\shuff{5pt}{3.5pt}}{\shuff{3pt}{2.6pt}}{\shuff{3pt}{2.6pt}}}}
\def\shuffle{{\, \shuffl \,}}
\begin{document}

\title{Cumulants, free cumulants and half-shuffles}

\author{Kurusch Ebrahimi-Fard}
\address{ICMAT,
		C/Nicol\'as Cabrera, no.~13-15, 28049 Madrid, Spain.
		On leave from UHA, Mulhouse, France}
         \email{kurusch@icmat.es}         
         \urladdr{www.icmat.es/kurusch}

\author{Fr\'ed\'eric Patras}
\address{Laboratoire J.-A.~Dieudonn\'e
         		UMR 6621, CNRS,
         		Parc Valrose,
         		06108 Nice Cedex 02, France.}
\email{patras@math.unice.fr}
\urladdr{www-math.unice.fr/$\sim$patras}

\date{26-Jan-2015}

\voffset=10ex

\begin{abstract}

Free cumulants were introduced as the proper analog of classical cumulants in the theory of free probability. There is a mix of similarities and differences, when one considers the two families of cumulants. Whereas the combinatorics of classical cumulants is well expressed in terms of set partitions, the one of free cumulants is described, and often introduced in terms of non-crossing set partitions. The formal series approach to classical and free cumulants also largely differ.

It is the purpose of the present article to put forward a different approach to these phenomena. Namely, we show that cumulants, whether classical or free, can be understood in terms of the algebra and combinatorics underlying commutative as well as non-commutative (half-)shuffles and (half-)unshuffles. As a corollary, cumulants and free cumulants can be characterized through linear fixed point equations. We study the exponential solutions of these linear fixed point equations, which display well the commutative, respectively non-commutative, character of classical, respectively free, cumulants.

\end{abstract}

%%%%%%%%%%%%%%%%%%%%%%%%%%%%%%%%%%%%%%%%%%%%%%%%

\maketitle


\tableofcontents

%%%%%%%%%%%%%%%%%%%%%%%%%%%%%%%%%%%%%%%%%%%%%%%%

\section{Introduction}
\label{sect:intro}

D.~Voiculescu introduced in the 1980s the theory of free probability \cite{DVoi}. In this theory the classical concept of probabilistic independence is replaced by the algebraic notion of freeness, i.e., the absence of relations. Briefly, the definition of a non-commutative probability space consist of a pair $(A,\phi)$, where $A$ is a complex algebra with unit $1_A$. The map $\phi$ is a $\mathbb{C}$-valued linear form on ${A}$, such that $\phi(1_{A}) = 1$. The elements of $A$ play the role of random variables, while the map $\phi$ should be considered as the expectation map, similar to classical probability theory. Let $I$ be a set of indices, and ${B}_i$, for $i\in I$, be subalgebras of $A$, containing the unit. The family of algebras ${B}_i$, $i\in I$, will be called free if $\phi(a_1\cdots a_n)= 0$ every time $\phi(a_j)=0$ $\forall j=1,\dots,n$ and $a_j\in B_{i_j}$ for some indices $i_1\neq i_2\neq\dots\neq i_n$.

R.~Speicher introduced the notion of free cumulants as the proper analog of classical cumulants in the theory of free probability. See e.g. \cite{SpeicherNica}, the standard reference on the subject. There is a mix of similarities and differences between the two families of cumulants. Indeed, whereas the combinatorics of classical cumulants is naturally expressed in terms of set partitions, the one of free cumulants is described and often introduced in terms of non-crossing set partitions. The formal series approach to cumulants and free cumulants also largely differ.

It is the purpose of the present article to develop a different approach to the algebraic and combinatorial structures underlying free and classical cumulants. Namely, we show that cumulants, both classical and free, can be understood algebraically in terms of (co-)commutative and non-(co-)commutative (un-)shuffles. As a corollary, cumulants and free cumulants happen to solve linear fixed point equations. 

Our approach to free cumulants involves Hopf algebraic structures. The general idea of investigating the combinatorics of free probabilities using such structures is not new and was developed by M.~Mastnak and A.~Nica in their seminal work on the logarithm of the $S$-transform \cite{Mastnak}. However, it should be pointed out that our work differs from theirs in its scope as well as in the particular Hopf algebra structures under consideration. Although the relations between free moments and free cumulants appear in their work, e.g., through the relations between the $S$- and $R$-transforms, their work aims at understanding how multiplication of freely independent $k$-tuples in a non-commutative space is reflected in free multiplicative convolution on the corresponding distributions; it involves Hopf algebras isomorphic to the one of symmetric functions and higher dimensional generalizations thereof. We are interested instead in unraveling the very definition of free cumulants and construct a Hopf algebra structure directly out of the underlying non-commutative probability space. To this end, compare, e.g., our Definition \ref{def:coproduct} with \cite[sect. 3]{Mastnak}.

In the present article we focus on the moment/cumulant relationship from an algebraic point of view. We show that the aforementioned linear fixed point equations can be solved in terms of proper exponentials using the pre-Lie Magnus expansion \cite{dendeq,EFM}. At this level, the basic difference between classical and free cumulants can be described analogously to the case of scalar- versus matrix-valued linear initial value problems. Indeed, classical cumulants correspond to an exponential solution of a linear fixed point equation in a commutative setting, whereas free cumulants correspond to solutions in a non-commutative context.      

\medskip 

In the following $k$ denotes a ground field of characteristic zero. This is basically the interesting case, and a convenient hypothesis to avoid cumbersome distinctions. However, we point out that this assumption is not strictly necessary for all the results in the article. Indeed, many equations we will consider are defined and can be solved over the integers. We also assume  any $k$-algebra $A$ to be associative and unital, if not stated otherwise. The unit in $A$ is denoted $\un$. Identity morphisms are written $I$.

\vspace{0.4cm}

\noindent {\bf{Acknowledgements}}: The first author is supported by a Ram\'on y Cajal research grant from the Spanish government. The second author acknowledges support from the grant ANR-12-BS01-0017, Combinatoire Alg\'ebrique, R\'esurgence, Moules et Applications and from the ESI Vienna, where this work was partially realized. Support by the CNRS GDR Renormalisation is also acknowledged. We would like to thank the referees for helpful comments.

%%%%%%%%%%%%%%%%%%%%%%%%%%%%%%%%%%%%%%%%%%%%%%%%

\section{Shuffle algebras}
\label{section:shudend}

Recall first various classical results and definitions related to shuffle algebras. In the classical literature, shuffles refer to the (commutative) combinatorial shuffles arising from products of (functional) iterated integrals that also appear in the theory of free Lie algebras \cite{Reutenauer}. They refer, however,  as well to topological shuffles, the latter being non-commutative (they are commutative only up to homotopy). These notions can be traced back at least to the 1950's -- the period in which both families of shuffle products were axiomatized in the works of Eilenberg--MacLane and Sch\"utzenberger \cite{EM,schu}.

Since we will be interested mainly in the non-commutative case, we will use the name ``shuffle algebra''  to denote general, possibly non-commutative, shuffle algebras and refer explicitly to ``commutative shuffle algebras'' in the commutative case.

Recall the definition of a {\it{shuffle}}, or {\it{dendrimorphic}}\footnote{Note that we prefer the word ``dendrimorphic'' over the commonly used terminology, ``dendriform'', which fails to meet the standard criteria of name-giving in mathematics by mixing greek and latin roots. In general, the terminology of shuffles seems, both for historical and conceptual reasons, more natural than the one of dendrimorphic structures -- we tend therefore to favour it.} algebra. It is a $k$-vector space $D$ together with two bilinear compositions $\prec$ and $\succ$ (the left and right half-shuffle products) subject to three axioms
\begin{eqnarray}
	(a\prec b)\prec c  &=& a\prec(b \shuffle c)        	\label{A1}\\
  	(a\succ b)\prec c  &=& a\succ(b\prec c)   		\label{A2}\\
   	a\succ(b\succ c)   &=& (a \shuffle b)\succ c        	\label{A3},
\end{eqnarray}
where the bilinear product 
\begin{equation}
	a \shuffle b:=\ a \prec b + a \succ b. \label{dendassoc}       
\end{equation}
We call $\shuffle$ the shuffle product on $D$.

A {\it{commutative shuffle algebra}}, sometimes also called Zinbiel algebra (in reference to the Bloh--Cuvier \cite{bloh,cuv1,cuv2} dual notion of Leibniz algebra, but we will stick to the classical terminology), is a shuffle algebra, where the left and right half-shuffles are identified:
$$	
	x \succ y = y \prec x ,
$$
so that in particular the shuffle product $\shuffle$ is then commutative: $x\shuffle y=x\prec y+x\succ y=y\shuffle x$. The standard example of a commutative shuffle algebra is provided by the tensor algebra, $\overline{T}(V):=\oplus_{n \ge 0} V^{\otimes n}$, over a $k$-vector space $V$ equipped with the left half-shuffle product recursively defined by (the definition of the right half-shuffle follows from commutativity $x \succ y = y \prec x$):
$$
	x_1\otimes \cdots \otimes x_n \prec y_1\otimes \cdots \otimes y_m:=
			x_1\otimes (x_2 \otimes \cdots \otimes x_n\ \shuffle\ y_1\otimes \cdots y_m).
$$

The axioms (\ref{A1}-\ref{A3}) imply that any shuffle algebra is an associative algebra for the shuffle product (\ref{dendassoc}). This observation actually underlies the classical and celebrated abstract proof of the associativity of the topological shuffle products by Eilenberg--MacLane \cite{EM}. Let us mention that one could actually show that the axioms of shuffle algebras encode exactly products of topological simplices. This is due to the equivalence between the computation of these products and computations in symmetric group algebras, i.e., in the Malvenuto--Reutenauer Hopf algebra \cite{MR,Patgeom}, together with the property of the latter to be free as a shuffle algebra \cite{Foissy}. 

Let us introduce some useful notations. Let $L_{a \succ} \left( b \right) = a \succ b = R_{\succ b} \left( a \right) $. The axioms yield
$$	
	L_{a \succ} L_{b \succ} = L_{a \shuffle b \succ}, 
	\qquad\ 
	R_{\prec a} R_{\prec b} = R_{\prec b \shuffle a}.
$$

Recall that a left pre-Lie algebra \cite{Cartier3, Manchon} is a $k$-vector space $V$ equipped with a bilinear product $\vartriangleright$, such that for arbitrary $a,b,c \in V$ 
\begin{equation}
\label{pLrel}
	a \vartriangleright (b\vartriangleright c) - (a\vartriangleright b) \vartriangleright c
		= b \vartriangleright (a\vartriangleright c) - (b\vartriangleright a) \vartriangleright c.
\end{equation}
It implies that the bracket $[a,b]:= a \vartriangleright b - b \vartriangleright a$ satisfies the Jacobi identity. For several reasons, largely due to the general theory of integration encoded by Rota--Baxter algebras (see \cite{EPnew}), pre-Lie algebras play a key role, e.g., in the understanding of recursive equations such as Bogoliubov's counterterm formula in perturbative quantum field theory. The next lemma follows directly from the axioms (\ref{A1}-\ref{A3}) of shuffle products.

\begin{lem}
Let $D$ be a shuffle algebra. The product $\vartriangleright : D \otimes D \to D$ 
$$
	a \vartriangleright b := a \succ b - b \prec a 
$$
is left pre-Lie. We write its left action $L_{a \vartriangleright} \left( b \right) = a \vartriangleright b=L_{a \succ} - R_{\prec a}$.
\end{lem}

\noindent Note that $[a,b]= a \vartriangleright b - b \vartriangleright a = a \shuffle b - b \shuffle a$ for all $a,b \in D$. The pre-Lie product is trivial (null) on commutative shuffle algebras, since we then have $a\succ b=b\prec a$.

Shuffle algebras are not naturally unital. This is because it is impossible to ``split'' the unit equation, $\un \shuffle a=a\shuffle \un=a$, into two equations involving the half-shuffle products $\succ$ and $\prec$. This issue is circumvented by using the ``Sch\"utzenberger trick'', that is, for $D$ a shuffle algebra, $\overline D := D \oplus k.\un$ denotes the shuffle algebra augmented by a unit $\un$, such that
\begin{equation}
\label{unit-dend}
    a \prec \un := a =: \un \succ a
    \hskip 12mm
    \un \prec a := 0 =: a \succ \un,
\end{equation}
implying $a\shuffle\un=\un\shuffle a=a$. By convention, $\un \shuffle \un=\un$, but $\un \prec \un$ and $\un \succ \un$ cannot be defined consistently in the context of the axioms of shuffle algebras. 

The following set of left and right half-shuffle words in $\overline{D}$ are defined recursively for fixed elements $x_1,\ldots, x_n \in D$, $n \in \mathbb{N}$
 \allowdisplaybreaks{
\begin{eqnarray*}
    w^{(0)}_{\prec}(x_1,\ldots, x_n) &:=& \un =:w^{(0)}_{\succ}(x_1,\ldots, x_n) \\
    w^{(n)}_{\prec}(x_1,\ldots, x_n) &:=& x_1 \prec \bigl(w^{(n-1)}_\prec(x_2,\ldots, x_n)\bigr)\\
    w^{(n)}_{\succ}(x_1,\ldots, x_n) &:=& \bigl(w^{(n-1)}_\succ(x_1,\ldots, x_{n-1})\bigr)\succ x_n.
\end{eqnarray*}}
In case that $x_1=\cdots = x_n=x$ we simply write $x^{\prec{n}}:=w^{(n)}_{\prec}(x,\ldots, x)$ and $x^{\succ{n}} := w^{(n)}_{\succ}(x,\ldots, x)$.

In the unital algebra $\overline D$ both the exponential and logarithm maps are defined in terms of the associative product~(\ref{dendassoc})
\begin{equation}
\label{ExpLog}
	\exp^\shuffle(x):=\un + \sum_{n > 0} \frac{x^{\shuffle n}}{n!}  
	\quad\ {\rm{resp.}} \quad\ 
	\log^\shuffle(\un+x):=-\sum_{n>0}(-1)^n\frac{x^{\shuffle n}}{n}. 
\end{equation}
Notice that we do not consider convergence issues: in practice we will apply such formal power series computations either in a purely algebraic setting (formal convergence arguments would then apply), or when dealing with graded algebras (then the series will reduce to a finite number of nonzero terms when restricted to a given graded component).

It is also convenient to introduce the so-called ``time-ordered'' exponential
$$
	\exp^{\prec}( x) := \un + \sum_{n > 0} x^{\prec{n}}.
$$
It corresponds to the usual time-ordered exponential in physics, when the shuffle product is defined with respect to products of, say, matrix- or operator-valued iterated integrals. See for instance \cite[Sect. 1]{bp}, where the links between products of iterated integrals and the (so-called ``shifted'') shuffle product in the Malvenuto--Reutenauer Hopf algebra are detailed. In  \cite{EPnew} a detailed study of time-ordered exponentials from an abstract algebraic point of view is presented. 

Similarly, we also define $\exp^{\succ}( x):=\sum_{n \ge 0} x^{\succ(n)}.$ Notice that $X=\exp^{\prec}( x)$ and $Z=\exp^{\succ}( x)$ are respectively the formal solutions of the two linear recursions 
$$
	X=\un + x\prec X \quad\ {\rm{resp.}} \quad\ Z= \un +x\succ Z.
$$
Both the time-ordered exponential as well as the proper exponential map (\ref{ExpLog}) will be key ingredients in our approach to cumulants. This point of view paves the way to new formal results on the structure and combinatorics of cumulants. 

Let us show, for example, how the classical group-theoretical properties of the flow map for, say, matrix-valued linear differential equations, translate almost immediately into the computation of a multiplicative inverse of the time-ordered exponential:

\begin{lem}\label{inverse-shuffle}
Let $A$ be a shuffle algebra, and $\overline A$ its augmentation by a unit $\un$. For $x \in A$ we have
$$
	\exp^{\succ}(-x) \shuffle \exp^{\prec}(x) = \un.
$$
\end{lem}

\begin{proof}
Indeed, we see that
\begin{eqnarray*}
	\exp^{\succ}(-x) \shuffle \exp^{\prec}(x)- \un
	&=& \sum\limits_{n+m\geq 1}(-1)^n\big\{(x^{\succ n})\prec (x^{\prec m}) 
					+ (x^{\succ n})\succ (x^{\prec m})\big\}\\
	&=& \sum\limits_{n>0,m\geq 0}(-1)^n(x^{\succ n})\prec (x^{\prec m}) 
				+ \sum\limits_{n\geq 0,m> 0}(-1)^n(x^{\succ n})\succ (x^{\prec m}).
\end{eqnarray*}
Now, since $(-1)^n(x^{\succ n})\prec (x^{\prec m})=(-1)^n((x^{\succ n-1})\succ x)\prec (x^{\prec m})=(-1)^n(x^{\succ n-1})\succ (x^{\prec m+1})$, the proof follows.
\end{proof}

Another useful result follows from the computation of the composition inverse of the time-ordered exponential.

\begin{lem}\label{inverse}
Let $A$ be a shuffle algebra, and $\overline A$ its augmentation by a unit $\un$. For $x\in A$ and $X := \un +Y :=\exp^\prec (x)$, then
$$
	x=Y\prec \big(\sum\limits_{n\geq 0}(-1)^nY^{\shuffle n}\big).
$$
\end{lem}

\begin{proof}
We follow \cite{foipat}. From $X=\un + \sum_{n > 0}x^{\prec n}$, we get $X-\un=Y=x \prec X$. On the other hand, the (formal) inverse of $X$ {\it{for the shuffle product}} is given by $X^{-1}=\frac{1}{1+Y}=\sum_{k\geq 0}(-1)^kY^{\shuffle k}$. We finally obtain 
$$
	x = x\prec \un 
	  = x\prec (X\shuffle X^{ -1})
	  = (x\prec X)\prec X^{-1}
	  =Y\prec \big(\sum\limits_{n\geq 0}(-1)^nY^{\shuffle n}\big).
$$
\end{proof}

There is an abundance of literature on non-commutative shuffles and on associated identities. The interested reader is referred to, e.g., \cite{chapotonMM,dendeq,newdend} for further insights and examples.

%%%%%%%%%%%%%%%%%%%%%%%%%%%%%%%%%%%%%%%%%%%%%%%%

\section{Unshuffling the double bar construction}
\label{sect:HopfAlg}

The notion dual to the one of shuffle product, i.e., the unshuffle coproduct has been considered only recently from an abstract axiomatic point of view. It plays a key role in the seminal works of L.~Foissy, and especially in his proof of the Duchamp--Hivert--Thibon ``free Lie algebra''  conjecture. We refer to his work for further details \cite{Foissy}.

\begin{defn}
A counital unshuffle coalgebra (or counital  codendrimorphic coalgebra) is a coaugmented coalgebra $\overline C = C \oplus k.\un$ with coproduct
\begin{equation}
\label{codend}
	\Delta(c) := \bar\Delta(c) + c \otimes \un + \un \otimes c,
\end{equation}
such that on $C$, $\bar\Delta = \Delta_{\prec} + \Delta_{\succ}$ with 
\begin{eqnarray}
	(\Delta_{\prec} \otimes I) \circ \Delta_{\prec}   &=& (I \otimes \bar\Delta)\circ \Delta_{\prec}        	\label{C1}\\
  	(\Delta_{\succ} \otimes I) \circ \Delta_{\prec}   &=& (I \otimes \Delta_{\prec})\circ \Delta_{\succ} 	\label{C2}\\
   	(\bar\Delta \otimes  I) \circ \Delta_{\succ}         &=& (I \otimes \Delta_{\succ})\circ \Delta_{\succ}   \label{C3}.
\end{eqnarray}
\end{defn}

\noindent The maps $\Delta_{\prec}$ and $\Delta_{\succ}$ are called respectively left and right half-unshuffles.

We shall omit the definition of an unshuffle (or codendrimorphic) coalgebra. The latter is obtained by removing the unit, that is, $\bar\Delta$ is acting on $C$, and has a splitting into two half-coproducts, $\Delta_{\prec}$ and $\Delta_{\succ}$, which obey relations (\ref{C1}), (\ref{C2}) and (\ref{C3}).

\begin{defn}
An unshuffle (or codendrimorphic) bialgebra is a unital and counital bialgebra $\overline B=B \oplus k.\un$ with product $\cdot$ and coproduct $\Delta$. At the same time $\overline B$ is a counital unshuffle coalgebra with $\bar\Delta = \Delta_{\prec} + \Delta_{\succ}$. The following compatibility relations hold 
\begin{eqnarray}
	\Delta^+_{\prec}(a \cdot b)  &=& \Delta^+_{\prec}(a)  \cdot \Delta(b)      	\label{D1}\\
  	\Delta^+_{\succ}(a \cdot b)  &=& \Delta^+_{\succ}(a)  \cdot \Delta(b),    	\label{D2}
\end{eqnarray}
where
\begin{eqnarray}
	\Delta^+_{\prec}(a)  &:=& \Delta_{\prec}(a) + a \otimes \un     	\label{D3}\\
  	\Delta^+_{\succ}(a)  &:=& \Delta_{\succ}(a) + \un \otimes a.     	\label{D4}
\end{eqnarray}
\end{defn}

We introduce now the algebraic structures encoding the relation between cumulants and moments in free probability, as viewed from the point of view of (un)shuffle (co)products.

\smallskip

Let $A$ be an associative $k$-algebra. Define $T(A):=\oplus_{n > 0} A^{\otimes n}$ to be the nonunital tensor algebra over $A$. The full tensor algebra is denoted $\overline T(A):=\oplus_{n \ge 0} A^{\otimes n}$. Elements in $T(A)$ are written as words $a_1\cdots a_n \in T(A)$ (to avoid ambiguities we denote $a \cdots a\in A^{\otimes n}$ by $a^{\otimes n}$, and the product of the $a_i$s in $A$ is written $a_1 \cdot_A a_2$). The space $T(A)$, equipped with the concatenation product of words (for $w=a_1\cdots a_n$ and $w'=b_1\cdots b_m$, $w\cdot w':=a_1\cdots a_nb_1\cdots b_m$), is a graded algebra. The natural degree of a word $w=a_1\cdots a_n$ is $n$, and we write $w \in T_n(A)$.

We also set $T(T(A)):=\oplus_{n > 0} T(A)^{\otimes n}$, and use the bar-notation to denote elements $w_1 | \cdots | w_n \in T(T(A))$, $w_i \in T(A)$, $i=1,\ldots,n$. The algebra $T(T(A))$ is equipped with the concatenation product. For $a= w_1 | \cdots | w_n$ and $b=  w_1' | \cdots | w_m'$ we denote their concatenation product in $T(T(A))$ by $a|b$, that is, $a|b := w_1 | \cdots | w_n | w_1' | \cdots | w_m'$. This algebra is multigraded, $T(T(A))_{n_1,\ldots ,n_k}:=T_{n_1}(A)\otimes \cdots \otimes T_{n_k}(A)$, as well as graded, $T(T(A))_n:=\bigoplus\limits_{n_1+ \cdots +n_k=n}T(T(A))_{n_1,\ldots ,n_k}$. Similar observations hold for the unital case, $\overline T(T(A))=\oplus_{n \ge 0} T(A)^{\otimes n}$, and we will identify without further comments a bar symbol such as $w_1|1|w_2$ with $w_1|w_2$ (formally, using the canonical map from $\overline T(\overline T(A))$ to $\overline T(T(A))$).

When $A$ is commutative (or graded commutative in the sense of algebraic topology), then $T(T(A))$ is classically involved in the definition of the double bar construction on $A$. This is a differential graded algebra structure appearing in homological algebra as well as in the study of $K(\Pi,n)$ spaces -- the latter can be seen as the very motivation underlying the Eilenberg--MacLane study of shuffle products in \cite{EM}. See, e.g., \cite{FPHH} for a modern account. The terminology ``double bar'' refers to the fact, that one may represent tensors $a_1\otimes \cdots \otimes a_n$ using bars, $a_1| \cdots |a_n$, instead of using the word notation. The representation of elements in $T(T(A))$ would then involve double bars. We point out that the combinatorial operations we are going to define and study on $T(T(A))$ are {\it{different}} from the classical structures existing on the double bar construction, even for a commutative algebra $A$.

\medskip

Given two (canonically ordered) subsets $S \subseteq U$ of the set of integers $\bf N$, we call connected component of $S$ relative to $U$ a maximal sequence $s_1, \ldots , s_n$ in $S$ such that there are no $ 1\leq i < n$ and $u \in U$, such that $s_i < u < s_{i+1}$. In particular, a connected component of $S$ in $\bf N$ is simply a maximal sequence of successive elements $s,s+1,\ldots ,s+n$ in $S$.

Consider a word $a_1\cdots a_n \in T(A)$. For $S:=\{s_1,\ldots, s_p\} \subseteq [n]$, we set $a_S:= a_{s_1} \cdots a_{s_p}$ (resp. $a_\emptyset:=1$). Denoting $J_1,\ldots,J_k $ the connected components of $[n] - S$, we also set $a_{J^S_{[n]}}:= a_{J_1} | \cdots | a_{J_k}$. More generally, for $S \subseteq U \subseteq [n]$, set  $a_{J^S_U}:= a_{J_1} | \cdots | a_{J_k}$, where the $a_{J_j}$ are now the connected components of $U-S$ in $U$.

\begin{defn} \label{def:coproduct}
The map $\Delta : T(A) \to \overline T(A) \otimes  \overline T(T(A))$ is defined by
\begin{equation}
\label{HA}
	\Delta(a_1\cdots a_n) := \sum_{S \subseteq [n]} a_S \otimes  a_{J_1} | \cdots | a_{J_k}
					   =\sum_{S \subseteq [n]} a_S \otimes a_{J^S_{[n]}}.
\end{equation} 
The coproduct is then extended multiplicatively to all of $\overline T(T(A))$
$$
	\Delta(w_1 | \cdots | w_m) := \Delta(w_1) \cdots \Delta(w_m),
$$
with $\Delta(\un):= \un \otimes \un$.
\end{defn}

\begin{thm} \label{thm:HA}
The graded algebra $\overline T(T(A))$ equipped with the coproduct \eqref{HA} is a connected graded non-commutative and non-cocommutative Hopf algebra. 
\end{thm}

\begin{proof}
By construction, $\overline T(T(A))$ is a graded algebra, and the map (\ref{HA}) respects the graduation and is both multiplicative and counital. It remains to show that $\Delta$ is coassociative. Note that the multiplicativity of $\Delta$ implies that it is enough to check the property on elements of $T(A)$.

We get:
\begin{eqnarray*}
	(\Delta \otimes I) \circ \Delta (a_1\cdots a_n) &=& (\Delta \otimes I)(\sum\limits_{U \subseteq [n]} a_U \otimes a_{J^U_{[n]}})\\
									   &=& \sum\limits_{S \subseteq U \subseteq [n]} a_S \otimes a_{J^S_U} \otimes a_{J^U_{[n]}} = (I\otimes \Delta)\circ\Delta (a_1\cdots a_n).
\end{eqnarray*}
\end{proof}

The crucial observation is that coproduct \eqref{HA} can be split into two parts as follows. On $T(A)$ define the {\it{left half-coproduct}} by
\begin{equation}
\label{HAprec+}
	\Delta^+_{\prec}(a_1 \cdots a_n) := \sum_{1 \in S \subseteq [n]} a_S \otimes a_{J^S_{[n]}},
\end{equation}
and
\begin{equation}
\label{HAprec}
	\Delta_{\prec}(a_1 \cdots a_n) := \Delta^+_{\prec}(a_1 \cdots a_n) - a_1 \cdots a_n \otimes \un. 
\end{equation}
The {\it{right half-coproduct}} is defined by
\begin{equation}
\label{HAsucc+}
	\Delta^+_{\succ}(a_1 \cdots a_n) := \sum_{1 \notin S \subset [n]} a_S \otimes a_{J^S_{[n]}}
\end{equation}
and
\begin{equation}
\label{HAsucc}
	\Delta_{\succ}(a_1 \cdots a_n) := \Delta^+_{\succ}(a_1 \cdots a_n) -  \un \otimes a_1 \cdots a_n.
\end{equation}
Which yields $\Delta = \Delta^+_{\prec} + \Delta^+_{\succ}$, and 
$$
	\Delta(w) = \Delta_{\prec}(w) + \Delta_{\succ}(w) + w \otimes \un + \un \otimes w.
$$
This is extended to $T(T(A))$ by defining
\begin{eqnarray*}
	\Delta^+_{\prec}(w_1 | \cdots | w_m) &:=& \Delta^+_{\prec}(w_1)\Delta(w_2) \cdots \Delta(w_m) \\
	\Delta^+_{\succ}(w_1 | \cdots | w_m) &:=& \Delta^+_{\succ}(w_1)\Delta(w_2) \cdots \Delta(w_m). 
\end{eqnarray*}

\begin{thm} \label{thm:bialg}
The bialgebra $\overline T(T(A))$ equipped with $\Delta_{\succ}$ and $\Delta_{\prec}$ is an unshuffle bialgebra. 
\end{thm}

\begin{proof}
From $\Delta(a| b)=\Delta(a) \Delta(b)$, we get 
$$
	\Delta_\prec^+(a | b)	=\Delta_\prec^+(a)\Delta(b),\ \Delta_\succ^+(a| b)
					=\Delta_\succ^+(a) \Delta(b).
$$
We know that the coproduct $\Delta$ is coassociative. For an element $w_1| \cdots |w_k \in T(T(A))$, let us write $w_i=a_{s_1^i}\cdots a_{s_{n_i}^i}$ with $S_i=\{s_1^i,...,s_{n_i}^i\}.$ We get:
\begin{eqnarray}
	(\Delta\otimes I)\circ\Delta (w_1|\cdots |w_k)
	&=& (I\otimes \Delta)\circ\Delta (w_1|\cdots |w_k) \nonumber\\
 	&=& \sum\limits_{X_i\subseteq T_i\subseteq S_i}(a_{X_1}|\cdots |a_{X_k})\otimes (a_{J^{X_1}_{T_1}}|\cdots 
	|a_{J^{X_k}_{T_k}})\otimes (a_{J^{T_1}_{S_1}}|\cdots |a_{J^{T_k}_{S_k}}).		\label{eqtechn}
\end{eqnarray}
Applying $\Delta_\prec$ instead of $\Delta$ to $w_1|\cdots |w_k$ amounts to limiting the range of variation of the $T_i$ in 
$$
	\Delta (w_1|\cdots |w_k)=\sum\limits_{T_i\subseteq S_i}(a_{{T_1}}|\cdots |a_{{T_k}})\otimes 
	(a_{J^{T_1}_{S_1}}|\cdots |a_{J^{T_k}_{S_k}})
$$
by requiring $1\in T_1$ and $\coprod T_i\not=\coprod S_i$. Similarly for higher order compositions of half-coproducts.

Eventually, we get 
\begin{itemize}
 \item $(\Delta_\prec\otimes I)\circ\Delta_\prec (w_1|\cdots |w_k)$ and $(I\otimes \bar\Delta)\circ\Delta_\prec (w_1|\cdots |w_k)$ are equal, and both are obtained by restricting the domain of the summation operator in (\ref{eqtechn}) to the $X_i,T_i,S_i$ such that $1\in X_1$, $\coprod X_i\not=\coprod T_i\not=\coprod S_i$.
 
 \item $(\Delta_\succ\otimes I)\circ\Delta_\prec (w_1|\cdots |w_k)$ and $(I\otimes \Delta_\prec)\circ\Delta_\succ (w_1|\cdots |w_k)$ are equal, and both are obtained by restricting the domain of the summation operator in (\ref{eqtechn}) to the $X_i,T_i,S_i$ such that $1 \in T_1, 1\notin X_i$, $\coprod X_i\not=\coprod T_i\not=\coprod S_i$.

\item $(\bar\Delta\otimes I)\circ\Delta_\succ (w_1|\cdots |w_k)$ and $(I\otimes \Delta_\succ)\circ\Delta_\succ (w_1|\cdots |w_k)$ are equal and both obtained by restricting the domain of the summation operator in (\ref{eqtechn}) to the $X_i,T_i,S_i$ such that $1\in S_1, 1\notin T_1$, $\coprod X_i\not=\coprod T_i\not=\coprod S_i$.
\end{itemize}
\end{proof}

%%%%%%%%%%%%%%%%%%%%%%%%%%%%%%%%%%%%%%%%%%%%%%%%

\section{Convolution and characters}
\label{sect:ConvChar}

Recall that the ultimate purpose of free probability theory is the study of linear forms on $T(A)$. However, this is equivalent to the study of linear forms that are multiplicative maps $\Phi$ on $T(T(A))$. This observation motivates the present section; the link with cumulant-moments relations in free probability will be made precise in the next section.

The following proposition is the natural generalization to unshuffle bialgebras of the fact that the convolution product equips the space of linear endomorphisms of a classical Hopf algebra with an associative algebra structure \cite{Cartier2}. We refer to \cite{foipat} for an application of these ideas to the study of the structure of commutative shuffle bialgebras.

Indeed, recall that the set of linear maps, $Lin(T(T(A)),k)$, is a $k$-algebra with respect to the convolution product defined in terms of the coproduct (\ref{HA}), i.e., for $f,g \in Lin(T(T(A)),k)$
$$
	f * g := m_k \circ (f \otimes g) \circ \Delta,
$$ 
where $m_k$ stands for the product map in $k$. We define accordingly the left and right half-convolution products:
$$
	f \prec g := m_k \circ (f\otimes g)\circ \Delta_\prec ,
$$
$$
	f \succ g := m_k \circ (f\otimes g)\circ \Delta_\succ .
$$

\begin{prop}
The space $({\mathcal L}_A:=Lin(T(T(A)),k), \prec, \succ)$ is a shuffle algebra.
\end{prop}

We recall its proof: for arbitrary $f,g,h\in {\mathcal L}_A$,
$$
	(f\prec g)\prec h	=m_k\circ ((f\prec g)\otimes h)\circ\Delta_\prec 
				=m_k^{[3]}\circ (f\otimes g\otimes h)\circ(\Delta_\prec\otimes I)\circ \Delta_\prec,
$$
where $m_k^{[3]}$ stands for the product map from $k^{\otimes 3}$ to $k$. Similarly 
\begin{eqnarray*}
	f\prec (g\shuffle h)&=&m_k\circ (f\otimes (g\shuffle h))\circ \Delta_\prec \\
				&=&m_k^{[3]}\circ (f\otimes g\otimes h)\circ (I\otimes \overline\Delta)\circ\Delta_\prec,
\end{eqnarray*}
so that the identity $(f\prec g)\prec h=f\prec (g\shuffle h)$ follows from $(\Delta_\prec\otimes I)\otimes\Delta_\prec =(I\otimes \overline\Delta)\circ\Delta_\prec$, and similarly for the other identities characterizing shuffle algebras.

As usual, we equip the shuffle algebra ${\mathcal L}_A$ with a unit, that is, in $\overline{{\mathcal L}_A}:={\mathcal L}_A\oplus k.\un \cong Lin(\overline T(T(A)),k)$, where in the last isomorphism the unit $\un \in \overline{{\mathcal L}_A}$ is identified with the augmentation map $e \in Lin(\overline T(T(A)),k)$ -- the null map on $T(T(A))$ and the identity map on $T(A)^{\otimes 0}\cong k$. That is, for an arbitrary $f$ in ${\mathcal L}_A$,
$$
	f\prec e=f=e\succ f,\ \ e\prec f=0=f\succ e.
$$

Let now $\phi$ be a linear form on $T(A)$. It extends uniquely to a multiplicative linear form $\Phi$ on $T(T(A))$ by setting
$$
	\Phi(w_1| \cdots |w_n):=\phi(w_1) \cdots \phi(w_n),
$$
(or to a unital and multiplicative linear form on $\overline T(T(A))$). Conversely any such multiplicative map $\Phi$ gives rise to a linear form on $T(A)$ by restriction of its domain.

This motivates the following definition, which generalizes to unshuffle bialgebras the classical link between characters and infinitesimal characters in the theory of classical Hopf algebras. For the latter, we refer to \cite{EGP}, where the equivalence between the two families of characters is studied in detail.

\begin{defn}
A linear form $\Phi \in \overline{{\mathcal L}_A}$ is called a character if it is unital, $\Phi(\un)=1$, and multiplicative, i.e., for all $a,b \in \overline T(T(A))$
$$
	\Phi(a|b)=\Phi(a)\Phi(b).
$$
A linear form $\kappa\in \overline{{\mathcal L}_A}$ is called infinitesimal character, if $\kappa(\un)=0$, and if for all $a,b\in T(T(A))$ 
$$
	\kappa(a|b)=0.
$$
\end{defn}

We write $Ch(\kappa)$ for the obvious extension of  a linear form on $T(A)$ (e.g.~the restriction to $T(A)$ of an infinitesimal character) to a character, defined by $Ch(\kappa)(\un):=1$, $Ch(\kappa)(w_1| \cdots|w_k):=\kappa(w_1) \cdots \kappa(w_k)$. Conversely, for an arbitrary $F\in \overline{{\mathcal L}_A}$, let us write $Res(F)$ for the infinitesimal character, which is defined as the restriction of $F$ to $T(A)$, and the null map on other tensor powers of $T(A)$ in $\overline T(T(A))$.

\begin{thm}\label{thm:Gg}
There exists another natural bijection $B$ between $G(A)$, the set of characters, and $g(A)$, the set of infinitesimal characters on $\overline T(T(A))$. More precisely, for $\Phi \in G(A), \exists ! \kappa  \in g(A)$ such that 
$$
	\Phi=e+\kappa\prec \Phi = \exp^{\prec}( \kappa),
$$
and conversely, for $\kappa\in g(A)$
$$
	\Phi:=\exp^{\prec}( \kappa)
$$
is a character.
\end{thm}

Let us use in the following the shortcut ``Hopf- or Sweddler-type'' notation $\Delta_\prec(w)=:w^{1,\prec}\otimes w^{2,\prec}$ (which is abusive, but its proper use should not result in wrong equations) .

\begin{proof}
We know from Lemma~\ref{inverse} that the implicit equation $\Phi= e+\kappa\prec \Phi=\exp^{\prec}(\kappa)$ has a unique solution $\kappa$ in $\overline{{\mathcal L}_A}$. Let us consider the infinitesimal character $\mu:=Res(\kappa)$, and let us show that $\mu$ also solves $\Phi=e+\mu\prec \Phi$; the first part of the Theorem will follow.

Indeed, for an arbitrary $a\in T(T(A)), a=w_1|\cdots|w_n$, notice first that by definition of the product $\prec$, and due to the vanishing of $\mu$ on any $T(A)^{\otimes k}$, for $k\not= 1$, we have:
$$
	(\mu\prec \Phi)(a)	=\mu(w_1^{1,\prec})\Phi(w_1^{2,\prec}|w_2|\cdots|w_n)
				=\kappa(w_1^{1,\prec})\Phi(w_1^{2,\prec}|w_2|\cdots|w_n).
$$
We immediately obtain, since 
$$
	\Phi(w_1)	=(e+\kappa\prec \Phi)(w_1)
			=\kappa(w_1^{1,\prec})\Phi(w_1^{2,\prec})
			=\mu(w_1^{1,\prec})\Phi(w_1^{2,\prec})
$$ 
that, for any $i>1$
$$
	\Phi (w_1|\cdots|w_n)	=\Phi(w_1)\Phi(w_2|\cdots|w_n)
					=\mu(w_1^{1,\prec})\Phi(w_1^{2,\prec}|w_2|\cdots|w_n))
					=(e+\mu\prec \Phi)(w_1|\cdots|w_n),
$$
from which the property follows.
 
Conversely:
$$
	\exp^{\prec}(\kappa)(w_1|\cdots|w_n)=(e+\kappa\prec\exp^{\prec}(\kappa))(w_1|\cdots|w_n)
							=\kappa(w_1^{1,\prec})\exp^{\prec}(\kappa)(w_1^{2,\prec}|\cdots|w_n).
$$
Assuming by induction that the property $\exp^{\prec}(\kappa)(w_1'|\cdots|w_k')=\exp^{\prec}(\kappa)(w_1') \cdots\exp^{\prec}(\kappa)(w_k')$ holds for elements $w_1'|\cdots|w_k'\in T(T(A))$ of total degree less than the degree of $w_1|\cdots|w_n$, yields
\begin{eqnarray*}
	\exp^{\prec}(\kappa)(w_1|\cdots|w_n)
				&=&\kappa(w_1^{1,\prec})\exp^{\prec}(\kappa)(w_1^{2,\prec})
								\exp^{\prec}(\kappa)(w_2)\cdots\exp^{\prec}(\kappa)(w_n)\\					
				&=&\exp^{\prec}(\kappa)(w_1)\exp^{\prec}(\kappa)(w_2) \cdots \exp^{\prec}(\kappa)(w_n).
\end{eqnarray*}
\end{proof}

%%%%%%%%%%%%%%%%%%%%%%%%%%%%%%%%%%%%%%%%%%%%%%%%

\section{Free Cumulants as infinitesimal characters}
\label{ssect:ncCM}

Recall now the definition of free cumulants \cite{Biane,speiLoth}, which underlies our previous developments. 

Let us start with the physical and probabilistic motivations for their introduction. The present approach appears to be particularly well fitted in this respect. We follow the seminal article by Neu and Speicher  \cite{NeuSpeicher}\footnote{We would like to thank R.~Speicher for pointing us to this work during his visit to ICMAT in 2013.}.

Consider a random evolution equation
$$
	\frac{d U}{d t}(t,t_0)=H(t)U(t,t_0),
$$
where $H(t)$ is a random operator, e.g., the stochastic interaction Hamiltonian associated to the modeling of an open system coupled to a heat reservoir \cite{NeuSpeicher}, or the one associated to a (randomized version) of the interaction Hamiltonian of adiabatic perturbation theory (see e.g.~\cite{BMP}). Such an equation is in general not solvable exactly, and, in practice, one has to simplify the problem (in our case by averaging over the various random solutions) and to eliminate degrees of freedom (by a suitable truncation process).

Writing $\langle U(t,s) \rangle$ for the averaging operator, we get the Picard--Dyson expansion
$$
	\langle U(t,s) \rangle = 1+ \sum\limits_{i=1}^\infty\;\; \idotsint\limits_{t \geq t_1\geq \dots\geq t_i\geq s} \langle H(t_1) \cdots H(t_i) \rangle dt_1\cdots dt_i.
$$
The Ansatz leading to free cumulants is then given by a master equation, which defines implicitly the free cumulants $k_{n+1}(t,t_1,\ldots,t_n)$ by
\begin{equation}\label{master}
	\frac{d}{dt} \langle U(t,s) \rangle = \sum\limits_{i=0}^\infty
	\;\; \idotsint\limits_{t \geq t_1\geq \dots\geq t_i\geq s}
	k_{i+1}(t,t_1,\ldots,t_i) \langle U(t,t_1) \rangle  \langle U(t_1,t_2) \rangle \cdots \langle U(t_i,s) \rangle dt_1 \cdots dt_i.
\end{equation}

Functional derivation shows that this last identity amounts to defining recursively the free cumulants by:
\begin{eqnarray*}
	\langle H(t)H(t_1) \cdots H(t_n) \rangle &=&\sum\limits_{r=0}^n \sum\limits_{\{i(1),\ldots ,i(r)\} \subset \{1,\ldots ,n\}} 
						k_{r+1}(t,t_{i(1)},\ldots ,t_{i(r)}) \langle H(t_1)\cdots H(t_{i(1)-1}) \rangle \cdots \\
					& &\hspace{4cm}\cdots \langle H(t_{i(r)+1})\cdots H(t_n) \rangle .
\end{eqnarray*}

It is well known that classical cumulants behave particularly well with respect to centered Gaussian processes, since in that case all cumulants vanish excepte for the second one. A striking property of free cumulants is, that the same property holds, i.e., all free cumulants vanish for centered processes excepte for the second one, for symmetric random matrix processes. That is, in the asymptotic regime ($N \longrightarrow \infty$), for $A(t)$ the symmetric $N \times N$ random matrix $(A_{ij}(t))$, with the $A_{ij}(t),\ i\leq j$ centered, Gaussian, independent and with the same covariance $\Gamma(t,s)= \langle A_{ij}(t)A_{ij}(s) \rangle$ and for the normalized expectation operator $\langle A(t_1)\cdots A(t_n)\rangle :=\frac{1}{N^{n/2+1}}\langle Tr[A(t_1)\cdots A(t_n)]\rangle $, one gets: $k_n(t_1,\ldots,t_n)=\Gamma(t_1,t_2)$ for $n=2$ and zero else.

There is a striking formal similarity between the definition of free cumulants for such a random process and the definition of the coproduct on $T(T(A))$ in the previous sections of the article. More generally, recall  the abstract definition of free cumulants.
A pair $(A,\phi)$, where $A$ is an associative $k$-algebra with unit and $\phi$ a linear form on $A$, is by definition a non-commutative probability space. The linear form is extended to $T(A)$, for all words $a_1\cdots a_n \in A^{\otimes n}$ 
$$
	\phi(a_1a_2a_3 \cdots a_n) := \phi(a_1\cdot_A a_2\cdot_A  a_3\cdot_A\ \cdots\  \cdot_A a_n).
$$
Viewing $a\in A$ as a non-commutative random variable, the moments of $A$ are defined by
$$
	m_n:=\phi(a^n)=\phi(a^{\otimes n}),
$$
whereas the free cumulants $k_n$ are obtained from the identity
\begin{equation}\label{freec}
	C(zM(z))=M(z),
\end{equation}
with $C(z):=1+\sum_{i\geq 1}k_nz^n,\ \ M(z):=1+\sum_{i\geq 1}m_nz^n$. Equivalently:
$$
	m_n=\sum\limits_{s=1}^n\sum\limits_{i_1 + \cdots + i_s = n-s} k_s m_{i_1} \cdots m_{i_s},
$$
where the $i_j$ run over the positive integers (i.e. the value $i_j=0$ is allowed).

Our main claim is that the fixed point equation (\ref{freec}) is a consequence of the fixed point equation $\Phi=e+\kappa\prec\Phi$ introduced in the previous section. Moreover, the same approach, properly abelianized, holds for classical cumulants, legitimizing in a new way the claim that free cumulants are a non-commutative version of classical cumulants.

To fix the ideas and illustrate concretely the half-shuffle approach, let us start with low-dimensional computations.
Let $\phi$ be the linear form on $T(A)$ associated to a non-commutative probability space $A$, and extended to $\overline T(T(A))$ multiplicatively, ${\Phi} : \overline T(T(A)) \to k$
$$
	\Phi(\un):=1,\ {\Phi}(w_1 | \cdots | w_m) := \phi(w_1) \cdots \phi(w_m). 
$$
Let  $\kappa  : \overline T(T(A)) \to k$ be the infinitesimal character solving the linear fixed point equation
\begin{equation}
\label{linEQ-Moments}
	{\Phi} = e + \kappa \prec {\Phi}. 
\end{equation}

We calculate a few simple examples. Let $a \in A\subset T(A)$. Then $\Delta^+_{\prec}(a)=a \otimes \un$, and hence, with ${\Phi}(\un)=1$
$$
	{\Phi}(a) = \kappa(a)=: k_1.
$$
Next we look at the two letters word $aa \in T_2(A)$. The left-coproduct reads $\Delta^+_{\prec}(aa)=aa \otimes \un + a \otimes a$, such that
$$
	\Phi(aa) = \kappa(aa) + \kappa(a)\kappa(a):=k_2 +k_1k_1.
$$
For $aaa \in T_3(A)$ the left-coproduct reads
$$
	\Delta^+_{\prec}(aaa) = aaa \otimes \un + a \otimes aa + 2 aa \otimes a, 
$$
such that
$$
	\Phi(aaa) = \kappa(aaa) + 3 \kappa(aa)\kappa(a) + \kappa(a) \kappa(a) \kappa(a) 
				=k_3 + 3k_2k_1 + k_1 k_1 k_1 .
$$
Let $aaaa \in T_4(A)$. The left-coproduct reads
$$
	\Delta^+_{\prec}(aaaa) = aaaa \otimes \un + a \otimes aaa + 2 aa \otimes aa + aa \otimes a|a + 3 aaa \otimes a. 
$$
This then gives
$$
	\Phi(aaaa) = \kappa(aaaa) 
					+ 4 \kappa(aaa)\kappa(a) 
						+ 2\kappa(aa) \kappa(aa) + 6\kappa(aa) \kappa(a)\kappa(a)
						+ \kappa(a)\kappa(a) \kappa(a) \kappa(a).
$$
We used that ${\Phi}(a|a)={\Phi}(a)\Phi(a)$. These equations coincide with the moments-cumulants relation for non-crossing partitions up to order four. More generally, we have

\begin{thm}\label{tim:freeprob}
Let $\phi: A \to k$ be a unital map, and $\Phi$ its extension to $\overline T(T(A))$  as above. Let the map $\kappa: \overline T(T(A))\to k$ be the infinitesimal character solving $\Phi = e + \kappa\prec \Phi$. For $a \in A$ we set $k_n:=\kappa(a^{\otimes n})$, $n \geq 1$ and $m_n:=\Phi(a^{\otimes n})=\phi(a^n)$, $n\geq 0$. Then
$$
	m_n=\sum\limits_{s=1}^n\sum\limits_{i_1 + \cdots + i_s = n-s} k_s m_{i_1} \cdots m_{i_s}.
$$
In particular, the $k_n$ identify with the free cumulants of $a \in (A,\phi)$.
\end{thm}

\begin{proof}
Indeed, notice first that subsets $\{1=s_1,\ldots,s_i\}=S \subset [n]$ are in bijection with sequences of (possibly null) integers of length $i$, $s_2 - s_1 - 1,\ldots, s_i - s_{i-1} -1, n-s_i$, and of total sum $n-i$. The nonzero terms of the sequence compute the lengths of the connected components of $[n]-S$ in $[n]$. We get
$$
	\Delta_\prec(a^{\otimes n})=\sum\limits_{i=1}^n a^{\otimes i} \otimes \sum\limits_{j_1+\cdots+j_i=n-i}
	a^{\otimes j_1}| \cdots |a^{\otimes j_i},
$$
with the convention that tensor powers $a^{\otimes 0}$ have to be ignored.

Applying this to $\Phi(a^{\otimes n})=(e+\kappa\prec \Phi)(a^{\otimes n})$, we get the expected identity 
$$
	m_n=\sum\limits_{s=1}^n\sum\limits_{i_1+ \cdots +i_s=n-s}k_sm_{i_1} \cdots m_{i_s}.
$$
 \end{proof}

Similar results hold for free cumulants over several variables. This can be deduced from the recursive definition of free cumulants following from the master equation (\ref{master}), but we prefer to detail the proof starting from the common definition of free cumulants in terms of non-crossing partitions.

Recall that a partition $\pi= \{P_1,\ldots ,P_k\},\ P_1\coprod \cdots \coprod  P_k=[n]$ is called non-crossing if and only if there are no $i,j,k,l$ in $[n]$, such that $i,k$ and $j,l$ belong to two disjoint blocks $P_{i_1}$, $P_{i_2}$ of the partition and $i<j<k<l$. We will assume that the $P_i$ are ordered according to their minimal element (${\rm{inf}}(P_1)=1 < {\rm{inf}}(P_2)< \cdots < {\rm{inf}}(P_k))$. The set of non-crossing partitions of $[n]$ is written $NC(n)$. For $\pi\in NC(n)$ as above, $a_1,\ldots ,a_n\in A$ and $\mu$ a linear form on $T(A)$, we write in general $\mu^\pi(a_1, \ldots ,a_n):= \prod_{i=1}^k\mu(a_{P_i})$. The generalized non-crossing cumulants $R(a_1,\ldots,a_n)$ associated to a unital map $\phi: A \to k$ are then the multilinear maps defined by the implicit equations (that can be solved recursively):
$$
	\phi(a_1\cdots a_n)=:\sum\limits_{\pi\in NC(n)}R^\pi(a_1,\ldots,a_n).
$$

\begin{thm}\label{gencumulant}
Let $\phi: A \to k$ be a unital map, and $\Phi$ its extension to $\overline T(T(A))$  as above. Let the map $\kappa: \overline T(T(A))\to k$ be the infinitesimal character solving $\Phi = e + \kappa\prec \Phi$. For $a_1,\ldots,a_n \in A$, we have: $\kappa(a_1\cdots a_n)=R(a_1,\ldots,a_n)$. That is, the infinitesimal character $\kappa$ computes the generalized non-crossing cumulants associated to $\phi$.
\end{thm}

\begin{proof}
Let us prove the theorem by induction on $n$. We assume that the multilinear map $R$ computing generalized non-crossing cumulants  agrees with the solution $\kappa$ of $\Phi=e+\kappa\prec \Phi$ on sequences $a_1,\ldots,a_i$ of length strictly less than $n$. We then have:
$$
	\Phi(a_1\cdots a_n)=\phi(a_1\cdots a_n)=\sum_{1\in S\subseteq  [n]}\kappa(a_S)\phi(a_{J_1})\cdots \phi(a_{J_{k(S)}}),
$$
where the $J_i$ are as usual the connected components of $[n]-S$, and $k(S)$ stands for the number of such components. We set $j_i:=|J_i|$. Using the induction hypothesis, we obtain:
$$
	\phi(a_1\cdots a_n)=\sum_{1\in S\subseteq  [n]}\kappa(a_S)\prod\limits_{i=1}^{k(S)}\Big(\sum\limits_{\pi_i\in NC(j_i)}\kappa^{\pi_i}(a_{J_i})\Big).
$$
However, it follows immediately from the definition of non-crossing partitions, that there is a canonical bijection between $NC(n)$ and the set of sequences $\{(S,\pi_1,\ldots,\pi_{k(S)})\}$, where $S$ runs over subsets of $[n]$ containing $1$ and the $\pi_i$ run over the non-crossing partitions of the connected components of $[n]-S$. Finally, we obtain:
$$
	\phi(a_1\cdots a_n)=\sum\limits_{\pi\in NC(n)}\kappa^\pi(a_1,\ldots,a_n),
$$
from which $\kappa=R$ on $T(A)$ follows, and hence the Theorem.
\end{proof}

%%%%%%%%%%%%%%%%%%%%%%%%%%%%%%%%%%%%%%%%%%%%%%%%%

\section{Cluster properties}
\label{sect:cluster}

An expected property of cumulant expansions is that they should behave in some sense meaningful from a physics point of view, that is, they should respect causality of the underlying system (by which we mean that independence properties should translate into the vanishing of corresponding cumulants). Neu and Speicher proved that this is indeed how free cumulants behave for random systems such as the ones studied at the beginning of the previous section, i.e., free cumulants associated to averages of solutions of random evolution equations. In this section, we show briefly that the argument presented in \cite{NeuSpeicher} holds, mutatis mutandis, on $T(T(A))$.

Let $\phi: T(A) \to k$ be a unital map, and $\Phi$ its multiplicative extension to $\overline T(T(A))$, $\Phi (w_1| \cdots |w_n):=\phi(w_1)\cdots \phi(w_n)$. Let the map $\kappa: \overline T(T(A))\to k$ be the infinitesimal character solving $\Phi = e + \kappa\prec \Phi$. We say that the cluster property holds for $B$ and $C$ two subsets of $A$, if and only if, for arbitrary elements  $b_1,\ldots ,b_n\in B$, $c_1,\ldots ,c_m \in C$,
$$
	\phi(b_1\cdots b_n c_1\cdots c_m)=\phi(b_1\cdots b_n)\phi(c_1\cdots c_m).
$$

\begin{prop}
For two subsets $B,C$ of $A$ satisfying the cluster property, the infinitesimal character $\kappa$ satisfies
$$
	\kappa(b_1\cdots b_n c_1\cdots c_m)=0
$$
for arbitrary $b_1,\ldots ,b_n \in B$ and  $c_1,\ldots ,c_m \in C$.
\end{prop}

\begin{proof}
Let us prove the theorem by induction on $k=n+m$. For $k=2$, that is, for $n=m=1$, we see that 
$$
	\kappa(b_1c_1)	=\phi(b_1c_1) - \kappa(b_1)\phi(c_1)
				=\phi(b_1)\phi(c_1) - \phi(b_1)\phi(c_1)
				=0.
$$

Assuming that the property holds for all $k<n+m$, we get from $\Phi - e = \kappa +\kappa\prec(\Phi-e)$, and for $a_1:=b_1,\ldots,a_n:=b_n \in B$, $a_{n+1}:=c_1,\ldots,a_{n+m}:=c_m\in C$,
$$
	X:=\kappa(b_1\cdots b_n c_1\cdots c_m) - \phi(b_1\cdots b_n c_1\cdots c_m)
		=- \sum\limits_{S \subsetneq [n+m] \atop 1 \in S} \kappa(a_S)\Phi(a_{J_{[n+m]}^S}).
$$
However, by the induction hypothesis, $\kappa(a_S)$ vanishes, whenever $S \cap \{n+1,\ldots ,n+m\} \not=\emptyset$, and we obtain:
$$
	X=-\sum\limits_{S \subset [n] \atop 1\in S} \kappa(a_S)\Phi(a_{J_{[n+m]}^S}),
$$
and, since the cluster property holds for $B,C$, we find
$$
	X=-\sum\limits_{S\subset [n] \atop 1\in S} \kappa(a_S)\Phi(a_{J_{[n]}^S})\phi(a_{n+1}\cdots a_{n+m}).
$$
Finally, $\phi(a_1\cdots a_n)=\sum_{S\subset [n], 1\in S}\kappa(a_S)\Phi(a_{J_{[n]}^S})$ and 
$$
	\kappa(b_1\cdots b_n c_1\cdots c_m)=\phi(b_1\cdots b_n c_1\cdots c_m) -\phi(b_1\cdots b_n)\phi(c_1\cdots c_m)=0,
$$
which concludes the proof of the proposition.
\end{proof}

%%%%%%%%%%%%%%%%%%%%%%%%%%%%%%%%%%%%%%%%%%%%%%%%

\section{Classical Cumulants from half-unshuffles.}
\label{sect:CM}

As the notion of unshuffle coalgebra is dual to the one of shuffle algebra, one can dualize the notion of a commutative shuffle algebra. It is an unshuffle coalgebra, in which $\Delta_\prec = \tau\circ\Delta_\succ$. Here $\tau$ denotes the twist map, $\tau(x \otimes y):=y\otimes x$. We will call an unshuffle bialgebra satisfying this property a  {\it{cocommutative unshuffle bialgebra}}\footnote{We feel that this terminology is more natural than the one of cozinbiel Hopf algebra used, e.g., in Fischer's thesis, to which we refer for details on the subject \cite{fischer}.}.
 
In the following, $A$ is an arbitrary commutative unital associative $k$-algebra equipped with a linear form $\phi : A \to k$ extended, as in the previous section, to a unital linear form on $\overline T(A)$ by $\phi(a_1a_2 \cdots a_n):=\phi(a_1 \cdot_A a_2\cdot_A  \cdots \cdot_A a_n)$. However, since our interest is oriented toward the moment/cumulant relation, the reader should have in mind for $A$ an algebra of scalar random variables admitting moments of all orders, and for $\phi={\mathbf E}$ the expectation operator. In the later case, the moments of $a\in A$ are given by $m_n:={\mathbf E}(a^n)$, and the series of cumulants $c_n$ is determined through
$$
	\sum\limits_{n\geq 1} c_n \frac{z^n}{n!}:= \log(\sum\limits_{n\geq 0}m_n \frac{z^n}{n!}).
$$
We will use, however, the equivalent definition of cumulants by means of the equations
\begin{equation}\label{cummom}
	m_n = c_n + \sum\limits_{m=1}^{n-1}{{n-1}\choose{m-1}} c_m m_{n-m}.
\end{equation}

The cocommutative unshuffle bialgebra structure on $\overline T(A)$ is defined by dualizing the one of a commutative shuffle algebra on the tensor algebra over $A$. We refer to \cite{Reutenauer} for details on the shuffle product on the tensor algebra. Concretely, the cocommutative coproduct is given by $\Delta^{\!\!\sh} : \overline T(A) \to \overline T(A) \otimes \overline T(A)$
$$
	\Delta^{\!\!\sh}(a_1 \cdots a_n):=\sum_{J \subseteq [n] } a_J \otimes a_{[n] - J}.
$$
This coproduct splits into left and right half-coproducts
\begin{equation}
\label{haprec+}
	\Delta^{\!\!\sh}_{\prec}(a_1 \cdots a_n) := \sum_{1 \in J \subseteq [n]} a_J \otimes a_{[n] - J}
\end{equation}
and 
\begin{equation}
\label{haprec}
	\Delta^{\!\!\sh}_{\succ}(a_1 \cdots a_n) := \tau\circ \Delta^{\!\!\sh}_{\prec}(a_1 \cdots a_n)
								   = \sum_{1 \notin J \subset [n]} a_J \otimes a_{[n] - J}.
\end{equation}
Together with the concatenation product, these maps define a structure of a cocommutative unshuffle bialgebra on $\overline T(A)$.

Notice that the relation $\Delta^{\!\!\sh}_{\succ}= \tau\circ \Delta^{\!\!\sh}_{\prec}$ implies that, for arbitrary $f,g\in Lin(\overline  T(A),k)$, we have:
$$
	f\prec g=g\succ f,
$$
with the usual conventions $f\prec g:=m_k\circ (f\otimes g)\circ \Delta^{\!\!\sh}_\prec$ respectively  $f\succ g:=m_k\circ (f\otimes g)\circ \Delta^{\!\!\sh}_\succ$,  so that $Lin(\overline T(A),k)$ is a commutative shuffle algebra for the left and right half-convolution products $\prec ,\succ$.

Let  now $\phi : \overline T(A) \to k$ be a unital map in $Lin(\overline T(A),k)$, and consider the linear fixed point equation
\begin{equation}
\label{classical-CM}
	\phi = e + c \prec \phi .  
\end{equation}

Here, the map $e : \overline T(A) \to k$ is the identity map on $A^{\otimes 0}$, and the null map on the other tensor powers of $A$. Let us calculate a few examples. Let $a \in T(A)$ be a single letter different from the empty word. Then $\Delta^{\!\!\sh}_{\prec}(a)=a \otimes \un$, and hence with $\phi(\un)=1$
$$
	\phi(a) = c(a):= c_1.
$$
Next we look at the word $aa \in T_2(A)$. The left-coproduct $\Delta^{\!\!\sh}_{\prec}(aa)=aa \otimes \un + a \otimes a$, such that
$$
	\phi(aa) = c(aa) + c(a)c(a):=c_2 + c_1c_1.
$$
For $aaa \in T_3(A)$ the left-coproduct reads
$$
	\Delta^{\!\!\sh}_{\prec}(aaa) = aaa \otimes \un + a \otimes aa + 2 aa \otimes a, 
$$
such that
$$
	\phi(aaa) = c(aaa) + 3 c(aa)c(a) + c(a) c(a) c(a) = c_3 + 3c_2c_1 + c_1 c_1 c_1 .
$$
Let $aaaa \in T_4(A)$. The left-coproduct reads
$$
	\Delta^{\!\!\sh}_{\prec}(aaaa) = aaaa \otimes \un + a \otimes aaa + 3 aa \otimes aa  + 3 aaa \otimes a. 
$$
This then gives
$$
	\phi(aaaa) = c(aaaa) 
					+ 4 c(aaa)c(a) 
						+ 3c(aa) c(aa) + 6c(aa) c(a)c(a)
						+ c(a)c(a)c(a)c(a)
$$
These identities coincide with the moments-cumulants relations up to order four.  

In general, we have $$
	\Delta^{\!\!\sh}(a^{\otimes n}) =\sum\limits_{S\subseteq [n]} a^{\otimes |S|} \otimes a^{\otimes n-|S|}
					=\sum\limits_{i=0}^n{n\choose i}a^{\otimes i}\otimes a^{\otimes n-i}
$$
and 
$$
	\Delta_\prec^\shuffle (a^{\otimes n})=\sum\limits_{1\in S\subset [n]}a^{\otimes |S|}\otimes a^{\otimes n-|S|}
						=\sum\limits_{i=0}^{n-1}{n-1\choose i}a^{\otimes i+1}\otimes a^{\otimes n-i-1},
$$
from which, with $c_n:=c(a^{\otimes n})$ and $m_n:=\phi(a^{\otimes n})=\phi(a^n)$, one finds that \eqref{classical-CM} gives the moment-cumulant relation
\begin{equation}
\label{classical-CM-Bell}
	m_n	= \phi(a^{n}) 
		= \sum_{j=0}^{n-1} {n-1 \choose j} c_{j+1} \phi(a^{n-j-1})
		= \sum_{j=0}^{n-1} {n-1 \choose j} c_{j+1} m_{n-j-1},    
\end{equation}

The same argument shows that the solution $c$ to the equation $\phi = e + c \prec \phi$ also computes the joint cumulants of a family $X_1,\ldots ,X_n$ of scalar random variables. This follows, e.g., from the fact that the generating series $G(\lambda_1,\ldots,\lambda_n)$ of joint cumulants of such a family is $\log {\mathbf E} (e^{\sum\limits_{i=1}^n\lambda_iX_i})$, so that combinatorial properties of joint cumulants reduce automatically to the ones of cumulants in a single variable (this argument does not hold for free cumulants, due to the non-commutativity of the algebras of random variables in free probabilities).

%%%%%%%%%%%%%%%%%%%%%%%%%%%%%%%%%%%%%%%%%%%%%%%%

\section{Exponentials and cumulants}
\label{sect:pL-Magnus}

Recall that the series of moments and cumulants are related by the logarithm and exponential maps. We explain why this nice relationship breaks down in the non-commutative framework of free probabilities. Notations are as in the previous sections.

Let us return to Theorem \ref{thm:Gg}, and recall from \cite{dendeq,EFM}, that the solution of the linear fixed point equation
\begin{equation}
\label{dendeq}
	\Phi = e + \kappa \prec \Phi
\end{equation} 
is also given in terms of the proper exponential (\ref{ExpLog}). Indeed, it can be shown that  
$$
	\Phi = \exp^\shuffle\!\!\big(\Omega'(\kappa)\big);
$$ 
where $\Omega'(\kappa)$ is called pre-Lie Magnus expansion and obeys the following recursive equation
$$
	\Omega'(\kappa) = \frac{L_{\Omega' \rhd}}{\exp(L_{\Omega' \rhd})-1}(\kappa)
            =\sum\limits_{m\ge 0} \frac{B_m}{m!}\ L^{m}_{\Omega' \rhd}(\kappa).
$$
Here, the $B_l$'s are the Bernoulli numbers.
Let us mention that $\Omega'(\kappa)$ can also be understood from the point of view of enveloping algebras of pre-Lie algebras \cite{chappat}.  We recall that $L_{a \rhd}(b):= a \rhd b = a \succ b - b \prec a$, where the product $a \rhd b$ satisfies the pre-Lie relation (\ref{pLrel}). See \cite{Manchon} for details.

Let us turn now to the commutative case, that  is, $a\succ b=b\prec a$, so that $a\rhd b=0$. These equations hold for instance in the commutative shuffle algebra $D:=Lin(\overline T(A),k)$ where $\overline T(A)$ is the cocommutative unshuffle bialgebra over a commutative algebra $A$. 
We use the notations of the previous section, and consider now the linear equation $\phi = e + c \prec \phi$.
Since, in general,
$$
	\Omega'(\kappa) = \kappa - \frac{1}{2} \kappa \rhd \kappa + \sum\limits_{m\ge 2} \frac{B_m}{m!}\ L^{m}_{\Omega' \rhd}(\kappa), 
$$
$\Omega'$ reduces in that case to the identity map, i.e., $\Omega'(c) = c$. Hence, in a commutative shuffle algebra the exponential solution of (\ref{dendeq}) reduces to 
$$
	\phi= \exp^\shuffle(c).
$$ 
This phenomenon is strictly analogous to what happens with ordinary scalar and matrix first order linear differential equations. Indeed, the first ones are solved by the exponential map, whereas the latter are solved by means of the Magnus formula, see e.g. \cite{EFM} for details.
 
 Let us focus now on the case $D=Lin(T(A),k)$ and show how this last formula expands combinatorially, which allows to recover the usual exponential computation of the generating series of moments from the one of cumulants. Notice first that, due to the set theoretical definition of the unshuffle coproduct $\Delta^\shuffle$ in $\overline T(A)$, for an arbitrary $a\in A$, we have
$$
	c\shuffle c(a^{\otimes n})	=m_k(c \otimes c)\Delta^\shuffle (a^{\otimes n})
						=\sum\limits_{p+q=n}{{n}\choose{p}}c(a^{\otimes p})c(a^{\otimes q})
						=\sum\limits_{p+q=n}{{n}\choose{p}}c_pc_q,
$$ 
and more generally 
$$
	c^{\shuffle k}(a^{\otimes n})=\sum\limits_{i_1+ \cdots +i_k=n} {{n}\choose{i_1,i_2,\ldots ,i_k}}c_{i_1} \cdots c_{i_k}.
$$
 We get
 $$
 	m_n=\phi(a^{\otimes n})=\sum\limits_{k\geq 1}\frac{1}{k!}\sum\limits_{i_1+ \cdots +i_k=n}
							{{n}\choose{i_1,i_2,\ldots ,i_k}}c_{i_1} \cdots c_{i_k},
$$
which is the degree $n$ component of the cumulant/moment relation (\ref{cummom}).
 
In conclusion, for a given non-commutative probability space $(A,\phi)$, the character $\Phi \in G(A) \subset Lin(\overline T(T(A)),k)$, which is defined as a multiplicative extension of the moment linear form $\phi$, can be written as the solution $\Phi = \exp^\shuffle\!\!\big(\Omega'(\kappa)\big)$ of the linear fixed point equation (\ref{linEQ-Moments}). The infinitesimal character $\kappa \in g(A)$ defines free cumulants. The classical analog of this situation is defined over $Lin(\overline T(A),k)$. The moment map is then given in terms of the cumulants map via the commutative exponential, $\Phi= \exp^\shuffle(c)$. From this perspective, the difference between free and classical cumulants-moments is once again displayed in the non-commutative and commutative $k$-algebras $Lin((\overline T(T(A)),k)$ and $Lin((\overline T(A),k)$, respectively.  \\

%%%%%%%%%%%%%%%%%%%%%%%%%%%%%%%%%%%%%%%%%%%%%%%%

\end{document}